\documentclass[12pt,a4paper,reqno]{amsart}

\usepackage{epsfig}
\usepackage{graphics}
\usepackage{multicol}
\usepackage{subfigure}

\usepackage{svn}
\SVN $Revision: 1854 $ 
\SVN $Date: 2011-08-30 16:01:09 +0200 (Die, 30. Aug 2011) $
\newif\ifprivate
\privatetrue
\usepackage{amsfonts,amsmath,amssymb}
\usepackage{url}

\allowdisplaybreaks

\newtheorem{theorem}{Theorem}

\newtheorem{lemma}[theorem]{Lemma}
\newtheorem{definition}[theorem]{Definition}
\newtheorem{proposition}[theorem]{Proposition}
\numberwithin{equation}{section}

\def\C{\mathbb{C}}

\newcommand{\calT}{\mathcal{T}}
\newcommand{\height}{\mathsf{height}}

\ifprivate  \fi

\newtheorem*{todoenvironment}{TODO}

\newcounter{epsilon}
\newcommand{\neweps}{\refstepcounter{epsilon}\varepsilon_{\arabic{epsilon}}}

\begin{document}
\title[The number of Huffman codes, compact trees, and sums of unit fractions]
{The number of Huffman codes, compact trees,\\ and sums of unit fractions}

\author{Christian Elsholtz}
\address{Institut f\"ur Mathematik A, Steyrergasse 30/II,
Technische Universit\"at  Graz,
A-8010 Graz, Austria}
\email{elsholtz@math.tugraz.at}
\author{Clemens Heuberger}
\address{Institut f\"ur Mathematik B, Steyrergasse 30/II,
Technische Universit\"at  Graz,
A-8010 Graz, Austria}
\email{clemens.heuberger@tugraz.at}
\author{Helmut Prodinger}
\address{
Helmut Prodinger,
Department of Mathematics,
University of Stellenbosch,
7602 Stellenbosch, South Africa}
\email{hproding@sun.ac.za}

\begin{abstract}
The number of ``nonequivalent'' 
Huffman codes of length $r$ 
over an alphabet of size $t$ has been studied frequently.
Equivalently, the number of ``nonequivalent'' complete $t$-ary trees
has been examined. We first survey the literature, 
unifying several independent approaches to the problem. 
Then, improving on earlier work 
we prove a very precise asymptotic result on the counting function, 
consisting of two main terms and an error term.

\end{abstract}

\maketitle
\section{Introduction}
\subsection{A problem in coding theory}
Let a source $S$ emit $r$ words $w_1, \ldots , w_r$ with probabilities
$p_1, \ldots , p_r$ respectively. Here 
$0\leq p_i\leq 1$ and $\sum_{i=1}^r p_i=1$.
For each word $w_i$ we assign a code word $c_i=c_i(w_i)$ over an 
alphabet of size $t$. Let $l_i$ denote the
length of the codeword $c_i$.
For a given source $S$, a compact code minimises the average 
length $\overline{l}=\sum_{i=1}^r p_i l_i$.
Huffman \cite{Huffman:1952} showed how to construct
a code with minimum average word length, given the word probabilities $p_i$. 
These Huffman codes are prefix-free, and can therefore be
decoded instantaneously. Moreover these codes can be found efficiently.

The Kraft-McMillan inequality states:
For an alphabet of size $t$ and a source that emits $r$ words, 
a necessary and sufficient condition
for the existence of an instantaneous code with code word lengths 
$l_1, \ldots , l_r$ is that
\begin{equation}{\label{kraft-inequality}}
 \sum_{i=1}^r \frac{1}{t^{l_i}} \leq 1.
\end{equation}
Moreover, for the existence of a uniquely decipherable code inequality 
(\ref{kraft-inequality}) is necessary.

Let us call a code \emph{compact} if it satisfies the Kraft equality:
\begin{equation}{\label{kraft-equality}}
 \sum_{i=1}^r \frac{1}{t^{l_i}} =1.
\end{equation}

When multiplying the equation by $t^{l_r}$ we observe 
that in a compact code the number of codewords of 
maximal length $l_r$ is divisible by $t$.
Also, if there are two distinct codewords starting with the same prefix 
$a_1 \ldots a_q$ but then continuing differently,
$a_1 \ldots a_q b_1\ldots $ and $a_1 \ldots a_q b_2\ldots$,
then all $t$ possible symbols must occur at position $q+1$.
In other words, if a sequence branches, it branches into all $t$ possible 
directions. This is the reason why it
is possible to model the situation by means of a rooted $t$-ary tree, which we
do below. As it is possible to arrive from a given Huffman code at a solution
of equation (\ref{kraft-equality}), and vice versa, 
to arrive from a solution to this equation at an
admissible Huffman code it is natural to consider all Huffman codes with the
same set of word lengths as ``equivalent''codes.

Example:
Let $t=3$.
Let the code consist of the codewords:
\[00, 010, 011, 012, 02, 1, 20, 21, 220, 221, 222.\]
The code can be nicely represented by the tree in Figure~\ref{tree1}.

\begin{figure}[h]
\caption{Rooted tree corresponding to the code
$\{00, 010, 011, 012, 02, 1, 20, 21, 220, 221, 222\}.$
}
{\label{tree1}}
\[ \includegraphics[width=
0.25\textwidth]{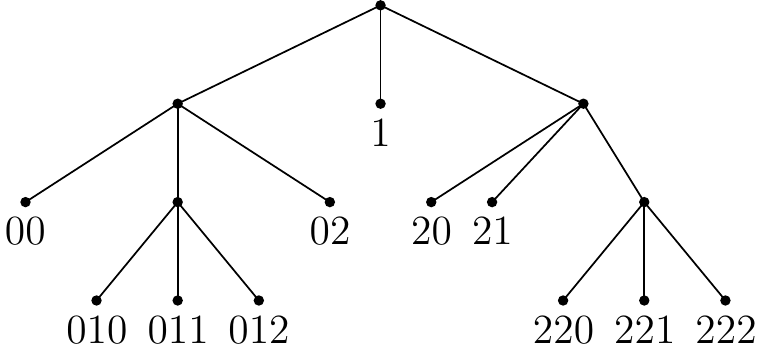}
\]
\end{figure}

Below we list a number of alternative ways
of defining our main object.
This reflects that the same type of question has been studied from
various points of view, often without being aware of the corresponding
 results expressed in a different mathematical language.

We use Kraft's equality as the basis for our first definition.
It stresses the number theoretic properties and was at the
origin of the Boyd's \cite{Boyd:1975} work.
\begin{definition}[Number theoretic definition]{\label{def:boydunitfraction}}
Let $f_t(r)$ denote the number of solutions of the equation
\[ \sum_{i=1}^r \frac{1}{t^{x_i}}=1,\]
where the $x_i$ are nonnegative integers and 
$0 \leq x_1 \leq \cdots \leq x_r$.
\end{definition}
For more information on other counting functions related 
to representations of one as a sum of unit fractions, 
see \cite{BrowningandElsholtz} and \cite{ChenandElsholtzandJiang}.

Collecting the number of words of the same length (corresponding to $x_i$ 
in the last definition), one arrives at an
alternative definition:
From our point of view, all codes with the same number of words of a given
length are equivalent. This suggests the following definition:
\begin{definition}[Huffman sequences]{\label{def:Huffmansequence}}
Let $t\geq 2$ and $r\geq 1$ be positive  integers. 
Let $f_t(r)$ denote the number of sequences of non-negative integers
\[(a_0, a_1, \ldots , a_l), 
\quad l\geq 0, \, a_l>0, \quad \sum_{i=0}^l a_i=r,\quad  
\sum_{i=0}^l \frac{a_i}{t^i}=1.\] 
\end{definition}

\subsection{Rooted trees}

Let us recall some vocabulary from graph theory:
A rooted tree is a connected cycle free graph, with one vertex being
distinguished (root). (We will draw it on the top, all other vertices below).
We say the tree is $t$-ary, if all those vertices, which are not the root,
 are either a leaf, that is an end
of a path from the root, or have one predecessor and 
$t$ children.  All non-leaves are called inner vertices. Note that the root is also
an inner vertex unless for the trivial tree of order one.
In other words, for the trees we
consider, the root has degree $t$, all other vertices either 
have degree 1 (leaf) or have degree $t+1$.

\begin{definition}[Canonical rooted tree]{\label{def:rootedtrees}}
A rooted tree is called canonical if its corresponding prefix code has the 
property that the lexicographic ordering of its words corresponds 
to a nondecreasing ordering of the word lengths. 

Let us say that two rooted $t$-ary trees are equivalent, if their number of
leaves at distance $i$ from the root is the same, for all $i$.
Let $f_t(r)$ denote the number of equivalence  classes of $t$-ary 
rooted trees with exactly $r$ leaves.
\end{definition}

Note that each equivalence class contains exactly one canonical tree.
Also, if the tree has $a_i$ leaves at distance $i$ from the root, 
then $\sum_i \frac{a_i}{t^i}=1$.  This follows inductively,
since a leaf at distance $i$ from the root, 
i.e. which contributes a weight $\frac{1}{t^i}$, 
can be split into $t$ children at distance $i+1$, of
weight $\frac{1}{t^{i+1}}$ each.
As these rooted $t$-ary trees correspond to a compact code, we also call these
trees ``compact trees''. 


Using Definition \ref{def:rootedtrees} one would for example replace the code 
\[\{00, 010, 011, 012, 02, 1, 20, 21, 220, 221, 222\}\]
by the following equivalent code:
\[\{0,10,11,12, 20, 210,211,212, 220, 221, 222\}.\]

The corresponding canonical rooted tree is in Figure \ref{tree2}.
\begin{figure}[h]
\caption{Canonical tree, corresponding to
$\{0,10,11,12, 20, 210,211,212, 220, 221, 222\}.$}
{\label{tree2}}
\[ \includegraphics[width=0.25\textwidth]{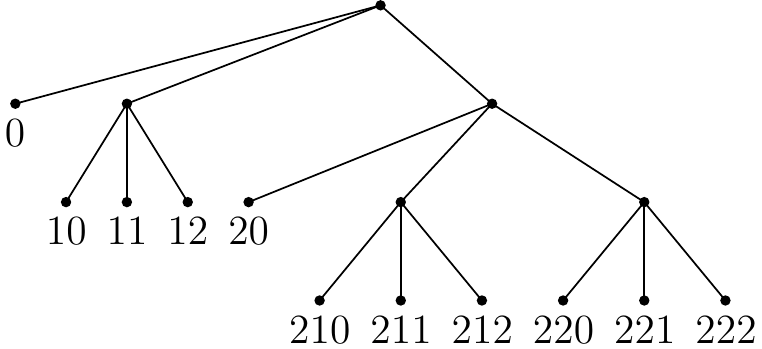}
\]
\end{figure}
In our usual way of drawing these diagrams, a canonical tree therefore
has the longer paths as far to the right hand side as possible.

\subsection{A problem on bounded degree sequences}
The number $a_i$ of code words of length $i$, 
or leaves at level $i$ is of course
bounded above by $t^i$. But there is no absolute bound on 
$\frac{a_i}{a_{i-1}}$. Let us study another sequence instead, namely 
$b_1=1, b_i=tb_{i-1}-a_{i-1}$, see 
Komlos, W. Moser and Nemetz \cite{KomlosanMoserandNemetz:1984} and 
Flajolet and Prodinger \cite{FlajoletandProdinger:1987}. 
The problems of counting these sequences
are equivalent to the earlier counting problem. 
For these sequences the ratios $\frac{b_i}{b_{i-1}}$ are bounded, 
which is why one may call these
sequences ``bounded degree sequences''.
Flajolet and Prodinger 
\cite{FlajoletandProdinger:1987} used this definition
when they counted level number sequences of trees.
\begin{definition}[Bounded degree]{\label{def:boundeddegree}}
Let $t\geq 2, r\geq 2$ be integers.
Let $f_t(r)$ denote the number of sequences
\[(b_1, \ldots , b_l),\quad l \geq 1, \quad b_1=1,\quad 1\leq b_i \leq t b_{i-1}
\quad (i=2, \ldots ,l),\quad \sum_{i=1}^l b_i=\frac{r-1}{t-1}.\]
For convenience we will later also use
$g_t(n)=f_t(1+n(t-1))$. (Here, one can think of $n=\frac{r-1}{t-1}$).

\end{definition}
A bijection between the last two definitions is as follows:
 Given a
canonical tree, we set $b_i$ to be the number of inner vertices at height 
$i-1$. Observe that the $b_i$ inner vertices guarantee that there are at most 
$t b_i$ vertices of any type (inner vertices or leaves) on the next level.

A very similar definition is due to Even and Lempel \cite{EvenandLempel:1972}. 
\begin{definition}[Proper words]{\label{def:properwords}}
Let $t\ge 2$ and $n\ge 1$ be integers. A word $u_1\ldots u_n$ over
the alphabet $\{0,1\}$ is said to be a proper word, if it can be written in the
form $u_1\ldots u_n=0^{c_0}10^{c_1}1\ldots0^{c_{l-1}}10^{c_l}$ such that $c_0=0$
and $0\le c_{i+1}\le t c_i+t-1$ holds for all $0\le i\le l-1$.
\end{definition}
Note that the sequence $c_i$ describes the lengths of the runs of consecutive
zeros. We note also that from the representation as a word of length $n$, we
immediately get $\sum_{i=1}^l c_i=n-l$.

To see that Definition \ref{def:properwords} is equivalent to
Definition~\ref{def:boundeddegree}, we simply note that the relations
$b_{i+1}=c_i+1$ and $n=\frac{r-1}{t-1}$ induce a bijection between the objects
counted in the two definitions. Even and Lempel~\cite{EvenandLempel:1972} also
give a combinatorial interpretation of this bijection (for $t=2$, but the
generalisation is straight-forward): essentially, for each $1$ in a proper word,
they replace a leaf of maximum height by an inner vertex with $t$ leaves as
successors; for each $0$, they replace a leaf of second-most height by an inner
vertex with $t$ leaves as successors.


We briefly mention some further approaches which investigate
 equivalent sequences.
Working on a different problem,
Minc \cite{Minc:1958} reduced it to the study of a binary bounded degree
sequence, Definition \ref{def:boundeddegree} above. 
Let $A$ be a free commutative entropic cyclic groupoid.
The number of elements of $A$ of a given degree turns out to satisfy 
the relation above. (For a full description we must refer to \cite{Minc:1958}).
The condition in Definition \ref{def:boundeddegree}
looks like a special partition function.
Andrews \cite{Andrews:1981} 
expanded on Minc's work, in particular studying generating functions.

A further problem, on lambda algebras $\Lambda_p$, has been related to these
sequences, see Tangora \cite{Tangora:1991}.
\subsection{An example}
As an example for these various definitions, 
let us compute $f_2(5)=3$ in the different forms.
Using Definition \ref{def:boydunitfraction}:
\[1=\frac{1}{2}+ \frac{1}{4}+\frac{1}{8}+\frac{1}{16}+\frac{1}{16}=
\frac{1}{2}+ \frac{1}{8}+\frac{1}{8}+\frac{1}{8}+\frac{1}{8}=
\frac{1}{4}+ \frac{1}{4}+\frac{1}{4}+\frac{1}{8}+\frac{1}{8}\]
is a complete list of all solutions.

Counting Huffman sequences (Definition \ref{def:Huffmansequence}) we count
 $(a_0,a_1, \ldots)$ where
$a_i$ is the number of occurrences of the fraction $\frac{1}{t^i}, i\geq 1$. 
Here with $t=2$ these sequences are:
\[(0,1,1,1,2),(0,1,0,4),(0,0,3,2).\]
Let us explicitly write down the compact Huffman codes.
\[C_1=\{0,10,110,1110,1111\},\ C_2=\{0,100,101,110,111\},\ 
C_3=\{00,01,10,110,111\}.\]

 The bounded degree sequences counted in
Definition \ref{def:boundeddegree} are $(1, 1,1,1),\ (1,1,2),\ (1,2,1)$.
The proper words in Definition \ref{def:properwords} are 
$(111), (110), (101)$.
The canonical trees (Definition \ref{def:rootedtrees}) are the following:


\begin{figure}[h]

\small\centering
\begin{minipage}[t]{3.8cm}
\includegraphics[width=0.55\textwidth]{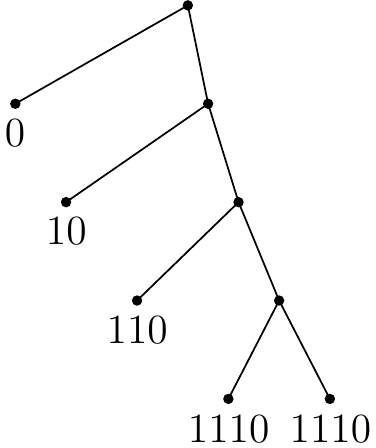}
\end{minipage}
\begin{minipage}[t]{3.8cm}
\includegraphics[width=0.55\textwidth]{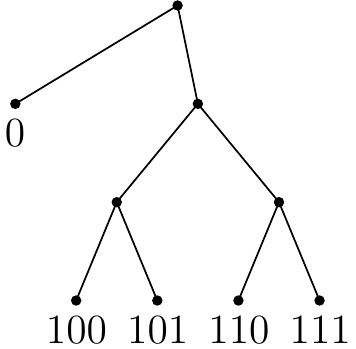}
\end{minipage}
\begin{minipage}[t]{3.8cm}
\includegraphics[width=0.55\textwidth]{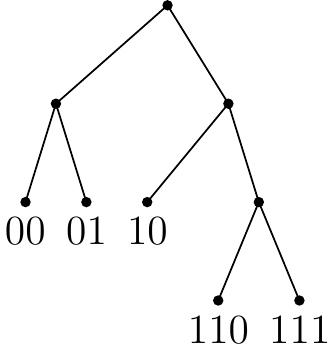}
\end{minipage}
\end{figure}

\subsection{An observation}
When evaluating $f_t(r)$, according to the Definition 
\ref{def:Huffmansequence} of Huffman sequences 
it suffices to investigate in which way a solution 
counted by $f_t(r-t+1)$ can be split.
Let $S_t(r)$ denote the set of all sequences counted by $f_t(r)$.
Generally, $(a_0, a_1, \ldots , a_i, \ldots, a_l)$ can be split into
$(a_0, a_1, \ldots, a_i -1, a_i+t, \ldots , a_l)$, whenever $a_i>0$.
Starting from a complete set of solutions, that is  $S_t(r-t+1)$, 
one only needs to branch each sequence at the last two positions, 
in order to compile a complete set of solutions, $S_t(r)$.
The reason for this is that all elements of $S_t(r)$ 
obtained from branching at any of the
earlier positions will be obtained 
from another member of $S_t(r-t+1)$ by branching at the last two positions.
Before we generally prove this let us look at an example. 
Let us determine $S_2(6)$, starting from the 
three elements of $S_2(5)=\{01112,0104,0032\}$:
\[
0111\vert 2 \rightarrow 0111\vert 12, \ 
011\vert 12 \rightarrow 011\vert 04, \ 
010\vert 4 \rightarrow 010\vert32,\ 
003\vert 2 \rightarrow 003\vert 12, \ 
00\vert 32 \rightarrow 00|24.\]

There is no need to consider
\[ 0\vert 1112 \rightarrow 0|0312, \text{ or }
01\vert 112 \rightarrow 01|032 \text{ or }
0\vert 104 \rightarrow 0|024,\]
as these are obtained otherwise.

To see this generally, let us consider
the step from $f_t(r-t+1)$ to $f_t(r)$:
If $(a_0, a_1, a_2, a_3, \ldots ,a_l)\in S_t(r-t+1)$, i.e. 
$\sum_{i=0}^l a_i=r-t+1$, with $a_l>0$, 
we need to check if  
$(a_0, a_1, \ldots, a_i-1, a_{i+1}+t, a_{i+2},\ldots , a_l)\in S_t(r)$
will be reached by branching an appropriate element
of $S_t(r-t+1)$ in any of the last two positions only.

Note that 
$(a_0, a_1, \ldots, a_i-1, a_{i+1}+t, a_{i+2}, \ldots, a_{l-1}+1, a_l-t)\in 
S_t(r-t+1)$.
Hence one reaches 
$(a_0, a_1, \ldots, a_i-1, a_{i+1}+t, a_{i+2},\ldots , a_{l-1}, a_l)\in S_t(r)$
 by branching in the last two positions only.
We may also observe that this gives a trivial upper bound of 
$f_t(r)\leq 2^{\frac{r-1}{t-1}}$.

Using the above  observation of branching at two positions only, 
Narimani and Khosravifard \cite{NarimaniandKhosravifard:2008} describe
a recursive algorithm to create all codes counted by $f_t(r)$.

The first terms of the sequence $f_2(r)$ are:
\[ t=2: 1, 1, 1, 2, 3, 5, 9, 16, 28, 50, 89, 159, 285, 510, 914, 1639,\ldots \]
The values of $f_3(r)$ are zero, whenever $r$ is even. The nontrivial part of
the sequence for odd $r$, that is $g_3(n)$ starts with 
\[ t=3: 1, 1, 1, 2, 4, 7, 13, 25, 48, 92, 176,\ldots\]
(see also \cite{PaschkeandBurkertandFehribach:2011}).
For general $t$, the sequence is only non-zero for $r=1+(t-1)n$.
For convenience one
examines $g_t(n)=f_t(1+n(t-1))$ instead, see Definition \ref{def:boundeddegree}.
For reference purposes we list the first values of the sequences
$g_t(n)$ in Table \ref{tab:g-sequences}. In these tables one can easily notice the observation above,
 $g_t(n)=f_t(r)\leq 2^{\frac{r-1}{t-1}}=2^n$. 

\begin{table}[htbp]
{\tiny{
\[\begin{array}{|r|rrrrrrrrrrrrrrrrrrrr|}\hline
t&1&2&3&4&5&6&7&8&9&10&11&12&13&14&15&16&17&18&19&20\\ \hline
2&1&1&1&2&3&5&9&16&28&50&89&159&285&510&914&1639&2938&5269&9451&16952\\
3&1&1&1&2&4&7&13&25&48&92&176&338&649&1246&2392&4594&8823&16945&32545&62509\\
4&1&1&1&2&4&8&15&29&57&112&220&432&848&1666&3273&6430&12632&24816&48754&95783\\
5&1&1&1&2&4&8&16&31&61&121&240&476&944&1872&3712&7362&14601&28958&57432&113904\\
6&1&1&1&2&4&8&16&32&63&125&249&496&988&1968&3920&7808&15552&30978&61705&122910\\
7&1&1&1&2&4&8&16&32&64&127&253&505&1008&2012&4016&8016&16000&31936&63744&127234\\
8&1&1&1&2&4&8&16&32&64&128&255&509&1017&2032&4060&8112&16208&32384&64704&129280\\
9&1&1&1&2&4&8&16&32&64&128&256&511&1021&2041&4080&8156&16304&32592&65152&130240\\
10&1&1&1&2&4&8&16&32&64&128&256&512&1023&2045&4089&8176&16348&32688&65360&130688\\ \hline
\end{array}
\]}}
\caption{Values of $g_t(n)$ for $2 \leq t \leq 10$ and $1\leq n \leq 20$.}
\label{tab:g-sequences}
\end{table}

The sequences $g_2(n), g_3(n)$ and $g_4(n)$ have been included into the 
OEIS (sequences  A002572, A176485 and A176503). (The latter two sequences 
only after the appearance of the Paschke et al. 
paper \cite{PaschkeandBurkertandFehribach:2011}.)

\subsection{The growth of $f_t(r)$}
As far as we are aware of,  Bende (1967) \cite{Bende:1967} 
and Norwood (1967) \cite{Norwood:1967}
were the first to examine the sequence 
$f_2(r)$, and they observed the connection to coding theory and trees.
(Minc's 1958 paper \cite{Minc:1958} was, of course, earlier 
but had less interest in the sequence itself.)
Bende asked about the asymptotic growth. 
Erd\H{o}s in his review of Bende's paper
(Mathematical Reviews) also wrote it is ``desirable'' to know the asymptotic.

The early 1970's saw a considerable number of contributions to the problem,
such as Boyd \cite{Boyd:1975}, Even and Lempel \cite{EvenandLempel:1972},
and Gilbert \cite{Gilbert:1971}.

A trivial upper bound for the number of rooted
canonical trees on $|V|$ vertices
is $2^{\binom{|V|}{2}}$. A much more precise bound is the number of all trees.
The number of binary trees on $|V|$ vertices is determined by the Catalan
numbers $\frac{1}{n+1} \binom{2n}{n}=O(4^n n^{-3/2})$ and the
number of non-isomorphic trees is asymptotically
$\sim C_2\, C_1^n n^{-5/2}$, where $C_1=2.955\ldots$ and 
$C_2=0.5349\ldots$, see Otter \cite{Otter:1948}.

A trivial lower bound comes from observing that Definition 
\ref{def:boundeddegree} shows that $f_2(r)\geq F_r$, where 
$F_r$ is the number of ways of partitioning $r-1$ into ones and twos.
It is known that this is the $r$-th Fibonacci number so that 
$f_2(r) \geq 0.4472\times 1.618329^r$ (for sufficiently large $r$).
Similarly, a lower bound on  $f_t(r)$, can be obtained by 
partitioning $r-1$ into 
$1$'s, $2$'s $\ldots$ and $t$'s. By means of the generating series
of $\frac{1}{1-z-z^2 - \cdots - z^t}$ and determining a real root of
the equation $1-z-z^2 - \cdots - z^t=0$
near $0.5$ the corresponding generalised Fibonacci number
$F_{t,r}$ can be shown to be about $c_t \rho_t^r$, where $\rho_t \approx 
2-\frac{1}{2^t-\frac{t}{2}}$, and $c_t$ is a positive constant. 
In the next section we will refine an analysis of this type considerably.

Boyd (1975) \cite{Boyd:1975},
Komlos, W.~Moser and Nemetz (1984) 
\cite{KomlosanMoserandNemetz:1984}, Flajolet
and Prodinger (1987) \cite{FlajoletandProdinger:1987},
all independently, gave an asymptotic:
\[ f_2(r) \sim R \rho^r,\]
where $R \approx 0.14185,\ \rho  \approx 1.7941471$.
Boyd and Flajolet and Prodinger additionally gave an error term:
$ f_2(r)=R \rho^r+ O(\tilde{\rho}^r),$
where Boyd proves $\tilde{\rho}=1.55$,
and Flajolet and Prodinger proved that this even holds for
$\tilde{\rho}=\frac{10}{7}$. 
Boyd, and Komlos, W.~Moser and Nemetz
 also study the case of more general $t$.
As noted before: 
as $f_t(r)$ is positive only for $r=1+n(t-1)$, one
examines $g_t(n)=f_t(1+n(t-1))$ instead.

In particular Komlos, Moser and Nemetz 
observed that $g_t(n) \sim K_t \rho_t^n$ with
$\rho_t \rightarrow 2$, as $t$ increases. 
Flajolet and Prodinger \cite{FlajoletandProdinger:1987} 
also refer to other areas, where the sequence $f_2(r)$ naturally occurs.

Building upon  \cite{FlajoletandProdinger:1987}, but not being aware of 
\cite{Boyd:1975} nor \cite{KomlosanMoserandNemetz:1984}, Tangora (1991)
\cite{Tangora:1991}  generalised the
results to prime values of $t$.

Another string of references follows from Gilbert's experimental 
observation that $f_2(r) \approx 0.148 (1.791)^r$, see \cite{Gilbert:1971}. 
The observation was based on
the values for $r\leq 30$, and is relatively close to the true asymptotic
$f_2 (r)\sim 0.1418\ldots (1.7941\ldots)^r$. However, these approximations have
been referred to in the more recent coding literature, see for example 
\cite{ParkerandRam:1999}, \cite{Savari:2009},
\cite{AbediniandKhatriandSavari:2010},
\cite{NarimaniandKhosravifard:2008},
 \cite{Khosravifard-Esmaeil-Saidi-Gulliver:2003} and
\cite{Khosravifard-Saidi-Esmaeil-Gulliver:2007}.

More recently Burkert (2010) \cite{Burkert:2010} and
Paschke, Burkert, Fehribach (2011) \cite{PaschkeandBurkertandFehribach:2011}
studied $f_2(r)$ and $f_t(r)$ respectively, unfortunately
with inferior results and unfortunately being unaware of the earlier 
work.\footnote{The oversights some decades 
ago can be easily explained due to the fact that 
the results were discovered independently by people with interests in number
theory, coding theory or graph theory.
Boyd's paper \cite{Boyd:1975} has a number theoretic title, 
the Komlos et al.~paper \cite{KomlosanMoserandNemetz:1984}
 a coding title and appeared in a less accessible journal.
Using standard tools such as
MathSciNet, Zentralblatt,  Google Scholar, 
Online Encyclopedia of Integer Sequences (OEIS) we found a
considerable corpus of literature referring to the result that 
$f_t(r) \sim K_t \cdot \rho_t^r$.}

In the results that we describe in detail in the next section, we 
state a rather precise asymptotic formula, with two main terms, and an error
term, which is \emph{exponentially} smaller. As an example, 
one finds an approximation
\[
f_2(n+1)\approx R \rho^{n+1} +R_2  \rho_2^{n+1},\]
with
\begin{align*}
\rho&=1.794147187541686,& \rho_2&=1.279549134726681,\\
R&=0.1418532020854094,& R_2&=0.0612410410312.
\end{align*}

Let us evaluate $f_2(50) \approx 699427308155.394\ldots$.
While the error analysis of Theorem \ref{theorem:asymptotics}
(below) gives an error of 
$|f_2(50)- (R \rho^{50} +R_2  \rho_2^{50})|
\leq 36.6\cdot 1.123^{50}\le 12092$,
the absolute error is much smaller
and, in this case, the above approximation
predicts the \emph{correct} value of $f_2(50)=699427308155$.


\subsection{A note on algorithms and complexity}
The question of the complexity of the evaluation of $f_2(r)$ is raised in 
Even and Lempel \cite{EvenandLempel:1972}. 
They give an algorithm to determine $f_2(r)$ in $O(r^3)$ additions.
This appears to be the only algorithm with analysis of its complexity.
They also state another
algorithm to give a complete list of the $f_2(r)$ elements.

Huffman, Johnson and Wilson \cite{HuffmanandJohnsonandWilson:2005} describe
another algorithm to give a complete list.

A tree based algorithm for generating the binary compact codes is described 
in \cite{Khosravifard-Esmaeil-Saidi-Gulliver:2003}.
Narimani and Khosravifard \cite{NarimaniandKhosravifard:2008} describe
a recursive algorithm to create all $t$-ary codes of length $r$
by those of length $r-t+1$.


\section{Results}

In the following, a tree will always be a $t$-ary rooted canonical tree.
The set of $t$-ary canonical trees is denoted by $\calT$.
The number of inner vertices (non-leaves) of a tree $T$ is
denoted by $n(T)$. Setting $c_{n}:=g_t(n)$ to be the number of trees $T\in\calT$ with
$n$ inner vertices, we are interested in the generating function
\begin{equation*}
  F(q)=\sum_{n\ge 0}c_n q^n=\sum_{T\in\calT} q^{n(T)}.
\end{equation*}

This generating function can be computed explicitly:
\begin{theorem}\label{theorem:generating-function}
Setting $[k]:=1+t+t^2+\cdots+t^{k-1}$, we have
  \begin{equation*}
    F(q)=\frac{\displaystyle\sum_{j=0}^{\infty}    (-1)^j q^{[j]}\prod_{i=1}^{j}\frac{q^{[i]}}{1-q^{[i]}}}%
    {\displaystyle\sum_{j=0}^\infty (-1)^{j} \prod_{i=1}^{j}\frac{q^{[i]}}{1-q^{[i]}}}.
  \end{equation*}
\end{theorem}

Using the generating function, we can give a very precise asymptotic expression
for $c_n$. In view of the numerous asymptotic approximations we would like to
point out that this is the first result containing two main terms and an
explicit error term.

\begin{theorem}\label{theorem:asymptotics}
For $t\ge 2$, the following holds:
 \begin{equation}\label{eq:main-asymptotic}
    c_n=g_t(n)=R\rho^{n+1}+R_2\rho_2^{n+1}+R_3 r_3^n{\neweps(t,n)},
  \end{equation}
Here $\rho> \rho_2> r_3$ and $R, R_2, R_3$ are positive real constants to be
specified below, and depending on $t$.
Here and below, $\varepsilon_j(\ldots)$, $j=1,\ldots$, denote
  real functions with
$|\varepsilon_j(\ldots)|\le 1$ for all valid values of the respectively indicated parameters.

For $t \geq 16$ we have
  \begin{align}\label{eq:rho-asymptotic}
    \rho&=2-\frac1{2^{t+1}} - \frac{t + 3}{2^{2  t +3}} - \frac{3t^{2} + 19  t +
      24}{2^{3t+6}} + \frac{0.28 t^3}{2^{4t}}\neweps(t),\\
    \rho_2&=1+\frac{\log 2}{t} - \frac{\log 2-\log^2 2}{2t^2} + \frac{4\log^32 + 3\log^22 +
      6\log2}{24t^3} \label{eq:second_pole_estimate_proof}\\&\qquad+ 
    \frac{2\log^42 + 54\log^32 - 27\log^22 -6\log2}{48t^4}
    +\frac{0.26}{t^5}\neweps(t),\notag\\
    r_3&=1+\frac{\log 2}{t}-\frac{\log
      2-\log^22}{2t^2},\label{eq:r_3-definition}\\
    R&=\frac{1}{8}+\frac{t - 2}{2^{t+5}}+ \frac{2  t^{2} +
3  t - 5}{2^{2t+7}}  + \frac{9  t^{3} + 45  t^{2} + 20  t -
68}{2^{3t+10}} + \frac{t^{4}}{50\cdot
2^{4t}}\neweps(t),\label{eq:dominant-residue}\\
    R_2&=\frac1{4t}-\frac{4\log 2 + 1}{8t^2}  +
    \frac{0.77}{t^3}\neweps(t)\label{eq:residue-second-pole},\\
R_3&=5 t^4.\label{eq:R_3definition}
  \end{align}
 
  For $3\le t\le 15$, \eqref{eq:main-asymptotic} holds with  
\eqref{eq:rho-asymptotic},
\eqref{eq:dominant-residue},
  \eqref{eq:residue-second-pole} and the values for $\rho_2, r_3$ and $R_3$
  given in Table~\ref{tab:special-values}.


  For $t=2$, \eqref{eq:main-asymptotic} holds with
  \eqref{eq:residue-second-pole} and the values for 
  $\rho$, $\rho_2$, $r_3, R$ and $R_3$ given in Table~\ref{tab:special-values}.
\end{theorem}


{\tiny{
\begin{table}[htbp]
  \centering
  \begin{equation*}
  \begin{array}{|c|l|l|l|l|l|l|}
\hline
t&\rho&\rho_2&r_3&R&R_2&R_3\\\hline
2&1.794147187541686&1.279549134726681&1.123&0.1418532020854094&0.0612410410312*&36.6\\
3&1.920712538405631*&1.211479378117327&1.098&0.1338681353605138*&0.05040725710011751*&39.0\\
4&1.964624757813775*&1.165158374565692&1.083&0.1305243270109503*&0.04239969309700251*&58.4\\
5&1.983293986764127*&1.134459698442781&1.074&0.1284678647212778*&0.03633182386516354*&70.7\\
6&1.991897175722647*&1.113019849812048&1.068&0.1271299952558400*&0.03168855397536632*&50.0\\
7&1.996015107731262*&1.097324075593615&1.063&0.1262776860399922*&0.02807600275247040*&59.6\\
8&1.998025544625657*&1.085389242111509&1.059&0.1257503987658994*&0.02520568904841775*&48.1\\
9&1.999017663916874*&1.076032488551186&1.056&0.1254328058843682*&0.02287594728315024*&24.0\\
10&1.999510161506312*&1.068511410911158&1.053&0.1252458295005635*&0.02094759256441895*&19.7\\
11&1.999755441055006*&1.062339511503337*&1.050&0.1251378340222618*&0.01932397366876184*&20.1\\
12&1.999877817773010*&1.057186165846774*&1.047&0.1250764428075050*&0.01793689446751572*&26.6\\
13&1.999938935019296*&1.052819586914068*&1.044&0.1250420050254539*&0.01673722535920120*&80.6\\
14&1.999969474502513*&1.049072853620226*&1.042&0.1250229006766309*&0.01568876914448585*&43.3\\
15&1.999984739115025*&1.045822904924682*&1.040&0.1250124013324635*&0.01476426249364319*&39.0\\ \hline
\end{array}
\end{equation*}

  \caption{Values for small values of $t$. Starred ($*$)  entries correspond
  to values satisfying the asymptotic estimates of
  Theorem~\ref{theorem:asymptotics}. The values could be given with much 
higher precision, there is some uncertainty about the last digit.}
  \label{tab:special-values}
\end{table}
}}

For simplicity the functions $\varepsilon_j$ can be thought of as $O(1)$
terms. Some of our proofs indeed depend on 
explicit values of the error bounds. For this reason we had to compute
absolute $O$-constants in any case, and decided to include these in the statement
of the theorem.

The asymptotic result focusses on the first and the second exponential terms
$\rho^{n+1}$ and $\rho_2^{n+1}$ and no effort has been made to improve the error
term $r_3^n$: note that for large $t$ it is not much smaller then the second order term
$\rho_2^{n+1}$. For Table \ref{tab:special-values}
the values $r_3$ have been improved by a computer calculation in comparison with
Equation (\ref{eq:r_3-definition}), also leading to a stronger value of the
constant $R_3$ in comparison with
(\ref{eq:R_3definition}). In principle, this type of improvement is possible
for any fixed $t \geq 16$ as well.

The asymptotic expansions of $\rho$, $\rho_2$, $R$ and $R_2$ can always be
refined by further iterating the fixed point equations in the proof of
Proposition~\ref{proposition:roots}. So for fixed $k$, we could refine the
estimates for $\rho$ and $R$ to a precision of $t^k2^{-tk}$ and the estimates
for $\rho_2$ and $R_2$ to a precision of $t^{-k}$.

\section{Generating Function}

This section is devoted to the proof of Theorem~\ref{theorem:generating-function}.

\begin{proof}[Proof of Theorem~\ref{theorem:generating-function}]
  In the proof of the theorem, we will actually consider more refined
  statistics in order to derive a functional equation for a more general
  generating function.

  The height of a vertex in a rooted tree is defined to be its distance from the
  root. So the root has height $0$. The height $\height(T)$ of a tree $T$ is defined to be the
  maximal height of its vertices.

  For a rooted tree $T$, we set $m(T)$ to be the number of leaves of maximum
  height of $T$.

  We will derive a functional equation for the generating function
  \begin{equation*}
    G(q,u)=\sum_{T\in\calT} q^{n(T)}u^{m(T)},
  \end{equation*}
  i.e., $u$ counts the number of leaves of maximal height and $q$ counts the
  number of
  inner vertices. By definition, we have $F(q)=G(q,1)$.
  
  To derive the functional equation for $G(q,u)$, we partition $\calT$
  with respect to the height and consider
  \begin{equation*}
    G_k(q,u)=\sum_{\substack{T\in\calT\\\height(T)=k}}q^{n(T)}u^{m(T)}.
  \end{equation*}
  Obviously, we have
  \begin{equation*}
    G(q,u)=\sum_{k\ge 0}G_k(q,u).
  \end{equation*}
  
  A tree $T$ of height $k$ corresponds to exactly $m(T)$ trees $T'_j$,
  $j\in\{1,\ldots,m(T)\}$, of height $k+1$: $T_j'$ arises from $T$ by replacing
  $j$ of the $m(T)$ leaves of maximum height by vertices with $t$ attached leaves. On the other
  hand, all trees $T'$ of height $k+1$ are uniquely described by this process.

  Thus we have
  \begin{equation}\label{eq:G_k_recursion}
    \begin{aligned}
      G_{k+1}(q,u)&=\sum_{\substack{T\in\calT\\\height(T)=k}}\sum_{j=1}^{m(T)} q^{n(T)+j}u^{jt}\\
      &=\sum_{\substack{T\in\calT\\\height(T)=k}}q^{n(T)}\cdot qu^t\cdot\frac{1-(qu^t)^{m(T)}}{1-qu^t}\\
      &=\frac{qu^t}{1-qu^t}\left(G_k(q,1)-G_k(q,qu^t)\right).
    \end{aligned}
  \end{equation}
  We have $G_0(q,u)=u$, so summing over all $k\ge 0$ yields
  \begin{equation}\label{eq:functional-equation-G}
    G(q,u)-u=\frac{qu^t}{1-qu^t}(G(q,1)-G(q,qu^t)).
  \end{equation}
  The generating function $G(q,u)$ is certainly convergent for $|u|\le1$ and $|q|<
  1/2$, as can be seen from \eqref{eq:G_k_recursion}.

  We now keep $q$ with $|q|< 1/2$ fixed and consider everything as a function
  of $u$ with $|u|\le 1$. We use the abbreviations $h(u)=qu^t/(1-qu^t)$ and $g(u)=G(q,u)$. We
  rewrite the functional equation \eqref{eq:functional-equation-G} as
  \begin{equation*}
    g(u)=u+h(u)g(1)-h(u)g(qu^t).
  \end{equation*}
  By iteration, we obtain
  \begin{align*}
    g(u)&=a_k(u)+b_k(u)g(1)+c_k(u)g(q^{[k+1]}u^{t^{k+1}}),\\
    a_k(u)&=\sum_{j=0}^k (-1)^j q^{[j]}u^{t^j}\prod_{i=0}^{j-1}h(q^{[i]}u^{t^i}),\\
    b_k(u)&=\sum_{j=0}^k (-1)^j \prod_{i=0}^{j}h(q^{[i]}u^{t^i}),\\
    c_k(u)&=(-1)^{k+1}\prod_{i=0}^{k}h(q^{[i]}u^{t^i})
  \end{align*}
  for $k\ge 0$. As $|h(u)|\le \frac{|q|}{1-|q|}<1$ holds for all $|u|\le 1$, the limits
  \begin{align*}
    a(u)&=\sum_{j=0}^{\infty}    (-1)^j q^{[j]}u^{t^j}\prod_{i=0}^{j-1}h(q^{[i]}u^{t^i}),\\
    b(u)&=\sum_{j=0}^\infty (-1)^j \prod_{i=0}^{j}h(q^{[i]}u^{t^i})
  \end{align*}
  exist and we have $\lim_{k\to \infty} c_k(u)g(q^{k+1}u^{t^{k+1}})=0$.

  Thus we obtained
  \begin{equation*}
    g(u)=a(u)+b(u)g(1).
  \end{equation*}
  Setting $u=1$ yields
  \begin{equation*}
    F(q)=G(q,1)=g(1)=\frac{a(1)}{1-b(1)}.
  \end{equation*}
\end{proof}

\section{Asymptotics}

We will use the following notations in order to work with the generating
function $F$:
\begin{align*}
    f_j(q)&=\frac{q^{[j]}}{1-q^{[j]}},\\
    N_K(q)&=\sum_{0\le k<K}(-1)^kq^{[k]} \prod_{j=1}^k f_j(q),&
    D_K(q)&=\sum_{0\le k<K}(-1)^k \prod_{j=1}^k f_j(q),\\
    N(q)&=\sum_{0\le k}(-1)^kq^{[k]} \prod_{j=1}^k f_j(q),&
    D(q)&=\sum_{0\le k}(-1)^k\prod_{j=1}^k f_j(q).
  \end{align*}
The quantities have been defined such that $F(q)=N(q)/D(q)$.

We intend to work with the finite sums $D_K$ and $N_K$ for fixed values of $K$,
so we need upper bounds for the approximation errors.

\begin{lemma}\label{lemma:approximation-errors}
  Let $K\ge 0$ and $|q|^{[K+1]}<1/2$. Then
  \begin{subequations}
    \begin{align}
      |N(q)-N_K(q)|&\le \left(\frac{1-|q|^{[K+1]}}{1-2|q|^{[K+1]}}\prod_{j=1}^K
        \frac1{1-|q|^{[j]}}\right) |q|^{[K]+\sum_{j=1}^K [j]},\label{eq:numerator-approximation}\\
      |D(q)-D_K(q)|&\le \left(\frac{1-|q|^{[K+1]}}{1-2|q|^{[K+1]}}\prod_{j=1}^K
        \frac1{1-|q|^{[j]}}\right) |q|^{\sum_{j=1}^K
        [j]}.\label{eq:denominator-approximation}
    \end{align}
  \end{subequations}
  These bounds are decreasing in $t$ and increasing in $|q|$.
\end{lemma}
\begin{proof}
  As $|f_j(q)|\le f_j(|q|)$ and  $f_j(|q|)$ is  decreasing in $j$, we have
  \begin{align*}
    |D(q)-D_K(q)|&\le \sum_{k=K}^\infty \prod_{j=1}^K f_j(|q|)\prod_{j={K+1}}^k f_j(|q|)\\
        &\le \prod_{j=1}^K f_j(|q|) \sum_{k=K}^\infty  f_{K+1}(|q|)^{k-K}\\
        &= \frac1{1-f_{K+1}(|q|)}\prod_{j=1}^K f_j(|q|),
  \end{align*}
  which, upon inserting the definition of $f_j$, yields
  \eqref{eq:denominator-approximation}. The approximation bound \eqref{eq:denominator-approximation} for the
  numerator follows along the same lines, we get an additional factor $q^{[K]}$.
\end{proof}

We will also need estimates for the derivative $D'(q)$:
\begin{lemma}\label{lemma:error-derivative}
  Let $t\ge 30$ and $q\in\C$ with $1/2\le |q|\le 1/r_3$, where $r_3$ is defined
  in \eqref{eq:r_3-definition}.
  
  Then
  \begin{equation*}
    |D'(q)-D'_4(q)|\le \frac1{2^{t^2}}.
  \end{equation*}
\end{lemma}
\begin{proof}
  Let $q=1/z$ with $r_3\le |z|\le 2$. Then $f_j(q)=f_j(1/z)=\frac1{z^{[j]}-1}$
  and $|f_j(q)|=1/|z^{[j]}-1|\le 1/(r_3^{[j]}-1)$. By estimating the relevant
  power series, we get
  \begin{subequations}
    \begin{align}
      r_3-1&\ge \frac{1}{2t},\notag\\
      r_3^{[2]}-1&= \exp\left((1+t)\log\left(1+\frac{\log 2}{t}-\frac{\log
            2-\log^22}{2t^2}\right)\right)-1
      \ge 1,\notag\\
      r_3^{[3]}-1&\ge 2^t,\label{eq:r_3_3_estimate}\\
      r_3^{[4]}-1&= 2^{t^2+t/2}.\label{eq:r_3_4_estimate}
    \end{align}
  \end{subequations}
  We have 
  \begin{align*}
    |D'(1/z)-D_4'(1/z)|&\le |z|\sum_{k=4}^\infty \prod_{j=1}^k
    f_j(1/|z|)\left(\sum_{j=1}^k \frac{[j]}{1-(1/|z|)^{[j]}}\right)\\
    &\le 2 \sum_{k=4}^\infty \frac t{2^{-1+t(k-1)/2+(k-3)t^2}}
    \left(4t+4\sum_{j=2}^k[j] \right)
    \le \sum_{k=4}^\infty\frac{k t^{k+1}}{2^{(k-3)t^2+t(k-1)/2-4}}\\&\le
    \frac12\sum_{k=4}^\infty\frac1{2^{t^2(k-3)}}\le \frac1{2^{t^2}}.
  \end{align*}

\end{proof}

The exponential growth of the coefficients $c_n$ of $F(q)$ is directly related
to the dominating pole $1/\rho$ of $F(q)$. So we now investigate the location
of the poles of $F(q)$.

\begin{proposition}\label{proposition:roots}
  Let $t\ge 2$. Then there are exactly two poles $1/\rho$ and $1/\rho_2$ of
  $F(q)$ with $|q|\le 1/r_3$, where $r_3$ has been
  defined in \eqref{eq:r_3-definition} (or Table~\ref{tab:special-values} for $t\in\{2,3\}$).

  Both $1/\rho$ and $1/\rho_2$ are simple poles of $F(q)$. The dominant pole
  $1/\rho$ of $F(q)$ is asymptotically given by \eqref{eq:rho-asymptotic} (or
  Table~\ref{tab:special-values} for $t=2$).
  
  The residue of $F(q)$ at $1/\rho$ is $-R$ where $R$ is asymptotically given 
  by \eqref{eq:dominant-residue} (or Table~\ref{tab:special-values} for $t=2$).

  The pole $1/\rho_2$ is given by \eqref{eq:second_pole_estimate_proof} (or
  Table~\ref{tab:special-values} for $2\le t\le 15$),
  the residue of $F(q)$ at $1/\rho_2$ is $-R_2$, where  $R_2$ is given in 
  \eqref{eq:residue-second-pole}.

  Finally, we have
  \begin{equation}\label{eq:F-bound-r_3}
    |F(q)|\le 5t^4
  \end{equation}
  for all $q$ with $|q|=1/r_3$.
\end{proposition}

The proof of Proposition~\ref{proposition:roots} relies on rewriting the
equation $D(q)=0$ into two fixed point equations, one for each of the two
poles. Inserting preliminary bounds into these fixed point equations improves
these bounds. This method is known as bootstrapping. The first pole is an
attracting fixed point of the first fixed point formulation, whereas the second
pole is a repellent fixed point of this first fixed point formulation. So we
need to take inverses in order to turn the second pole into an attracting fixed
point. However, inversion involves extracting a $(t+1)$-st root, so several
branches occur. Additional inequalities are required in order to decide which
branch to take. We repeatedly use power series estimates in order to get the
required inequalities. In order to sharpen these estimates, we assume that
$t\ge 30$.

\begin{proof}
  In the proof of this proposition, some more functions $\varepsilon_j(\ldots)$
  occur. We first allow complex values for the $\varepsilon_j(\ldots)$, it will
  later turn out that those occurring in Theorem~\ref{theorem:asymptotics} have
  only real values. 

  In the following, we consider the case $t\ge 30$.
  Assume that $1/z$ is a pole of $F(q)$ with $|z|\ge 1+a/t$ for some $2\ge a\ge \log
  2$. As $N(q)$ is holomorphic for $|q|<1$, cf.\ Lemma~\ref{lemma:approximation-errors}, $1/z$ must be a root of $D(q)$.
  Using $K=3$, we get
  \begin{equation*}
    0=1-\frac{1}{z-1}+\frac{1}{z-1}\frac{1}{z^{t+1}-1}+(D(1/z)-D_3(1/z)),
  \end{equation*}
  which is equivalent to
  \begin{equation}\label{eq:bootstrap-equation}
    2-z=\frac1{z^{t+1}-1}+ (z-1)(D(1/z)-D_3(1/z)).
  \end{equation}
  Taking absolute values, \eqref{eq:denominator-approximation} yields
  \begin{equation}\label{eq:bootstrap}
    2-|z|\le |2-z|\le \frac1{|z|^{[2]}-1}\left(1+\frac1{|z|^{[3]}-1}\cdot\frac1{1-\frac1{|z|^{[4]}-1}}\right).
  \end{equation}
  We have
  \begin{equation}\label{eq:r-2-estimate}
    \begin{aligned}
      |z|^{[2]}&\ge \left(1+\frac at\right)^{t+1}=\exp\left((t+1)\log\left(1+\frac at\right)\right)
      \ge\exp\left((t+1)\left(\frac at-\frac{a^2}{2t^2}\right)\right)\\
      &=\exp\left(a+\frac {a-a^2/2}t-\frac{a^2}{2t^2}\right)\ge\exp\left(a+\frac bt\right) 
      \ge e^a\left(1+\frac bt\right)
    \end{aligned}
  \end{equation}
  for $b=a-31a^2/60>0$. By \eqref{eq:r_3_3_estimate} and
  \eqref{eq:r_3_4_estimate}, we have
  \begin{equation}\label{eq:r-3-r-4-first-estimate}
    \frac{1}{|z|^{[3]}-1}\cdot \frac1{1-\frac1{|z|^{[4]}-1}}\le \frac{1.00001}{2^t}.
  \end{equation}
  Consider now the case $a=\log 2$. Then \eqref{eq:bootstrap}, \eqref{eq:r-2-estimate} and
  \eqref{eq:r-3-r-4-first-estimate} yield
  \begin{equation}\label{eq:first-bootstrap}
    2-|z|\le \frac1{1+\frac{2b}t}\left(1+\frac{1.00001}{2^t}\right)\le 1-\frac4{5t}.
  \end{equation}
  We conclude that $|z|\ge 1+\frac4{5t}$. So using now $a=4/5$, \eqref{eq:bootstrap}, \eqref{eq:r-2-estimate} and
  \eqref{eq:r-3-r-4-first-estimate} yield
  \begin{equation*}
    2-|z|\le \frac1{e^{4/5}-1}\left(1+\frac{1.00001}{2^t} \right)\le
    0.82
  \end{equation*}
  and therefore $|z|\ge 1.18$. Inserting this and \eqref{eq:r-3-r-4-first-estimate} in \eqref{eq:bootstrap} now
  yields
  \begin{equation*}
    2-|z|\le |2-z|\le  \frac1{1.18^{t+1}-1}\left(1+\frac{1.00001}{2^t} \right)\le \frac{0.86}{1.18^t}.
  \end{equation*}
  We conclude that $z=2+O(1.18^{-t})$. We now rewrite
  \eqref{eq:bootstrap-equation} as
  \begin{equation}\label{eq:bootstrap-final-equation}
    z=2-\frac1{z^{t+1}-1}+ O(2^{-t^2}).
  \end{equation}
  Inserting $z=2+O(1.18^{-t})$ in the right-hand side of
  \eqref{eq:bootstrap-final-equation}
  yields
  \begin{equation*}
    z=2-\frac1{(2+O(1.18^{-t}))^{t+1}-1}=2-\frac1{2^{t+1}}\left(1+O(t \, 1.18^{-t})\right).
  \end{equation*}
  We now repeat the process: We insert this estimate in the right-hand side of
  \eqref{eq:bootstrap-final-equation} and get a better estimate. After a few
  iterations (and taking care of all implicit constants), we finally get
  \eqref{eq:rho-asymptotic}. Inserting the lower and the upper bounds of
  \eqref{eq:rho-asymptotic} into $D_3(q)$ (and taking into account
  $D(q)-D_3(q)$), we see that $D(q)$ changes sign within the interval, so there
  is certainly a root $1/z$ of $D(q)$ fulfilling \eqref{eq:rho-asymptotic}.

  Inserting this asymptotic expression into $D'(q)$ and using Lemma~\ref{lemma:error-derivative}, we get
  \begin{equation}\label{eq:derivative-at-dominant-root}
    |D'(1/z)+4|\le 1.04 t 2^{-t}
  \end{equation}
  for $t\ge 30$. This shows that there is at most one zero of $D(1/z)$ within the
  bounds of the asymptotic expression~\eqref{eq:rho-asymptotic}: if there were
  two, say $1/z_1$ and $1/z_2$, then 

\begin{align*}
    4\left|\frac1{z_2}-\frac1{z_1}\right|&=
    \left|D(1/z_2)-D(1/z_1)+4\left(\frac1{z_2}-\frac1{z_1}\right)\right|\\
    &=\left|\int_{[1/z_1,1/z_2]}(D'(q)+4)\,dq\right|\le 1.04t2^{-t}\left|\frac1{z_2}-\frac1{z_1}\right|,
  \end{align*}
  which implies $1/z_1=1/z_2$. Here, we integrate over the straight line from
  $1/z_1$ to $1/z_2$. The estimate
  \eqref{eq:derivative-at-dominant-root} also shows that there can only be a
  simple root. Thus we have shown that the only root $1/z$ of $D$ with $|z|\ge
  1+\log2/t$ is a simple zero with $z$ as in \eqref{eq:rho-asymptotic}. The
  residue \eqref{eq:dominant-residue} follows upon inserting
  \eqref{eq:rho-asymptotic} into $N(1/z)/D'(1/z)$. Note that this also shows
  that the dominant zero of the denominator does not cancel out against a zero
  of the numerator.

  Now assume that $|D(1/z)|\le 1/t^3$ holds for some $z$ with $r_3\le |z|\le 1+{\log 2}/t$.
  Inserting these bounds into \eqref{eq:bootstrap}, we get
{\refstepcounter{epsilon}\label{epsilon_6a}
  \begin{equation}\label{eq:bound-2-r-3}
    |z-2|\le 1-\frac{\log2}{t} + \frac{4 \log^3 2 - 3\log^22 + 12\log 2}{12t^2}
    +\frac{1.5}{t^3}\varepsilon_{\ref{epsilon_6a}}(t,z)=: r'.
  \end{equation}
}  The intersection point with positive imaginary part of the circle of radius
  $1+\log 2/t$ centred at the origin with the circle of radius $r'$ centred
  at $2$ is denoted by $\xi$. We obtain
  \begin{equation*}
    \xi=1+\frac{4\log2+i\sqrt{\frac{16}3\log^32 - 4\log^22 + 16\log2} }{4t} +
\frac{2.23}{t^2}\neweps(t).
  \end{equation*}
  In particular, we have
  \begin{equation}\label{eq:z-1-difference-bound}
    |z-1|\le |\xi-1|\le \frac{1.14}t
  \end{equation}
  and
  \begin{equation}\label{eq:z-argument-bound}
    |\arg(z)|\le |\arg \xi|\le |\log \xi|\le \frac{1.18}t.
  \end{equation}

  As $|D(1/z)|\le1/t^3$, we have (after multiplication with $z-1$)
  \begin{equation*}
    0=z-2+\frac{1}{z^{t+1}-1}+\frac{2.01}{t^3}\neweps\label{epsilon_9}(t,z).
  \end{equation*}
  Solving for $z^{t+1}$ yields
  \begin{equation*}
    z^{t+1}=1+\frac1{2-z-\frac{2.01}{t^3}\varepsilon_{\ref{epsilon_9}}(t,z)}.
  \end{equation*}
  As $z=1+\frac{1.14}t\neweps(t,z)$ by \eqref{eq:z-1-difference-bound}, we obtain
  \begin{equation*}
    z^{t+1}=2+\frac{1.19}t\neweps\label{epsilon_11}(t,z).
  \end{equation*}
  We conclude that
  \begin{equation}\label{eq:inverse-bootstrap-with-argument}
    z=\exp\left(\frac{2\ell\pi i}{t+1}+\frac{1}{t+1}\log\left(2+\frac{1.19}t\varepsilon_{\ref{epsilon_11}}(t,z)\right)\right)
  \end{equation}
  for some integer $\ell$ with $-\frac{t+1}2<\ell\le \frac{t+1}{2}$. In
  particular, we have
  \begin{equation*}
    \arg z=\frac{2\ell\pi}{t+1}+\frac{1}{t+1}\Im\log\left(1+\frac{1.19}{2t}\varepsilon_{\ref{epsilon_11}}(t,z)\right),
  \end{equation*}
  which, in view of \eqref{eq:z-argument-bound}, implies $\ell=0$. Thus
  \eqref{eq:inverse-bootstrap-with-argument} simplifies to
  \begin{equation*}
    z=\exp\left(\frac{1}{t+1}\log\left(2+\frac{1.19}t\varepsilon_{\ref{epsilon_11}}(t,z)\right)\right)=
    1+\frac{\log2}t+\frac{1.63}{t^2}\neweps(t,z).
  \end{equation*}
  We may now repeat the argument a few times to finally obtain
  \begin{equation*}
    z=1+\frac{\log 2}{t} - \frac{\log 2-\log^2 2}{2t^2} + \frac{4\log^32 + 3\log^22 +
6\log2}{24t^3} + \frac{3.45}{t^4}\neweps(t,z).
  \end{equation*}
  Thus we have $|z|>r_3$. We have therefore shown that
  \begin{equation*}
    |D(q)|\ge \frac1{t^3}\qquad\text{for}\qquad |q|=1/r_3.
  \end{equation*}
  So we now assume that $D(1/z)=0$ for some $z$ with $r_3\le
  |z|<1+\log2/t$. Repeating the above steps with $1/t^3$ replaced by $0$ gives
  the slightly better bound $z=\rho_2$ with $\rho_2$ as in \eqref{eq:second_pole_estimate_proof}.
  
  Inserting the real upper and lower bounds implied by
  \eqref{eq:second_pole_estimate_proof} into $D_3(q)$ and taking the error
  $D(q)-D_3(q)$ into account shows that the sign of $D(q)$ changes sign in this
  interval, so there is a real root $1/z=1/\rho_2$ of $D(q)$ fulfilling
  \eqref{eq:second_pole_estimate_proof}.

  For the $z$ in \eqref{eq:second_pole_estimate_proof}, we get
  \begin{equation*}
    D'(1/z)=\frac{2}{\log 2} t^2+1.07t\neweps(t,z),
  \end{equation*}
  which implies that there is exactly one simple zero $1/z$ of $D(q)$ with $z$ fulfilling
  \eqref{eq:second_pole_estimate_proof}. By the same argument as above, this is the only zero
  $1/z$ with $r_3\le |z|<1+\log2/t$. Computing $N(1/z)/D'(1/z)$ finally yields
  the residue given in \eqref{eq:residue-second-pole}.

  We already know that $|D(q)|\ge 1/t^3$ for all $q$ with $|q|=1/r_3$. We also
  get $|N(q)|\le 5t$. This yields \eqref{eq:F-bound-r_3}.

  We now turn to the case $2\le t<30$. Here, the asymptotic estimates can be
  replaced by concrete numbers. All assertions have been proved using the 
  interval arithmetic built in in Sage~\cite{Stein-others:2011:sage-mathem}.
  First, we computed an estimate analogous to \eqref{eq:bound-2-r-3}. The
  corresponding neighbourhood of $2$ is subdivided in squares. Each of these
  squares is intersected with its image under \eqref{eq:bootstrap-equation}
  and the union of its images under the corresponding analogon to
  \eqref{eq:inverse-bootstrap-with-argument}. If this intersection is empty or
  the square has no point of absolute value at least $r_3$, the square is
  discarded. Otherwise, the square is replaced by the smallest square
  containing the mentioned intersection. If this does not yield sufficient
  progress, the squares have been ``bisected'' into four squares. After a
  certain number of operations, there are only two small regions which might
  contain a root. Estimating the derivative, we see that there is at most one
  root in each of these regions. As it is suspected that these roots are real,
  the real bisection method is employed to determine the roots with higher
  precision. The approximation errors $D(q)-D_K(q)$ can also be handled by
  adding the corresponding interval in the interval arithmetic.
\end{proof}

We are now able to prove Theorem~\ref{theorem:asymptotics}.

\begin{proof}[Proof of Theorem~\ref{theorem:asymptotics}]
  This is a consequence of singularity analysis
  \cite{Flajolet-Odlyzko:1990:singul}, cf.\ also
  \cite{Flajolet-Sedgewick:ta:analy}.

  In this simple case, this also follows from Cauchy's integral formula and the
  residue theorem (and Proposition~\ref{proposition:roots}):
  \begin{equation*}
\varepsilon_1(t,n)5t^4r_3^{n}=
\frac1{2\pi i}\oint_{|q|=1/r_3} \frac{F(q)}{q^{n+1}}\,dq=
-R\rho^{n+1}-R_2\rho_2^{n+1}+c_n.
  \end{equation*}
\end{proof}

\section{Acknowledgements}
C. Heuberger is supported by the Austrian Science Fund (FWF): S09606, that is part of the Austrian National Research Network “Analytic Combinatorics and Probabilistic Number Theory.”
This paper was partly written while C. Heuberger was a visitor at Stellenbosch University.

\bibliography{cheub}
\bibliographystyle{amsplain}

\end{document}